\documentclass[12pt]{article}

\usepackage{amsmath,amssymb,amsthm}
\usepackage{bbm}
\usepackage{enumerate}

\newtheorem{theorem}{Theorem}
\newtheorem{lemma}[theorem]{Lemma}
\newtheorem{proposition}[theorem]{Proposition}
\newtheorem{corollary}[theorem]{Corollary}

\title{Completely greedy coin sets}
\author{Andrew J. Young}
\date{}

\begin{document}

\maketitle

\begin{abstract}
We show that the binary coin set minimizes the number of coins needed to guarantee the ability to make change in any one transaction and its asymptotic uniform average cost is no worse than that of any completely greedy coin set.
\end{abstract}

\section{Introduction}

For our purposes, a \emph{coin set} is a set of positive integers containing $1$, and a \emph{sub-coin set} is a subset of a coin set that is also a coin set, i.e. it contains $1$.
When enumerated, a coin set is always written in increasing order and its $k$th element is a reference to this order.
Let $\mathcal{C}$ be a coin set.
An element $c \in \mathcal{C}$ is a \emph{coin}.
One is said to \emph{make change} at a positive integer $n$ by choosing a \emph{representation} in $\mathcal{C}$, collection of coins $c \in \mathcal{C}$ that sum to $n$, i.e. $n = \sum_{c \in \mathcal{C}} \alpha_{c} c$ for some $\alpha_{c} \in \mathbb{N}$.
A representation is \emph{minimal} if it uses the least possible number of coins, $\sum_{c \in \mathcal{C}} \alpha_{c}$ is minimized.
We say a coin set is \emph{greedy} if successively choosing the greatest lower bounding coin provides a minimal representation for every positive integer $n$.
A coin set $\mathcal{C}$ is said to be \emph{totally greedy} \cite{Cowen2008} if the first $k$ elements of $\mathcal{C}$ are greedy for all $1 \le k \le \lvert \mathcal{C} \rvert$.
A \emph{transaction} is a positive integer of interest.

One problem is that of \emph{optimal denomination}, the coin set minimizing the average coins per transaction.
Under the greedy and totally greedy settings the following pathological issue arises.
Suppose one is interested in making change for transactions in $\{1,2, \ldots, N \}$, for some $N$, then without any further restrictions the optimal denomination is $\{1,2, \ldots, N \}$.
To combat this pathology we introduce completely greedy coin sets.

A coin set is said to be \emph{completely greedy} if all of its sub-coin sets are greedy.
With this restriction the maximum number of coins in a coin set is at most logarithmic in $N$, see Lemma \ref{lemma_bin_lb} and its succeeding remarks.
Completely greedy coin sets are characterized by specific subsets of cardinality $3$ (Proposition \ref{prop_3pt_gc}), and greediness of subsets of cardinality $3$ can be quickly verified (Corollary \ref{cor_3pt}).
Moreover, the binary coin set, i.e. $\{1,2, \ldots, 2^{k}, \ldots \}$, is completely greedy, completely unique (Proposition \ref{prop_div_unique}), minimizes coin purse size (Proposition \ref{prop_coin_purse}), and is asymptotically optimal in some sense (Theorem \ref{theorem_icg_lb}).

While complete greediness is a rather strong condition, we believe it is justified by its resilience to the loss of elements of a coin set different from $1$.
More specifically, suppose a store makes change using some coin set $\mathcal{C}$.
For practical purposes, as making change is in general $NP$-hard \cite{Shallit2003}, a greedy coin set is desirable.
However, without complete greediness, this property may be lost if the store were to run out of some non-unit denominations.

\section{Minimal coin purse}

We begin by studying the minimum number of coins needed to guarantee the ability to make change in any one transaction.
To that end, given a coin set $\mathcal{C}$, an \emph{optimal coin purse for $\mathcal{C}$} is the smallest list of elements of $\mathcal{C}$ such that any one transaction can be represented using elements of that list.
Further minimizing over all possible coin sets provides a \emph{minimal coin purse}.
Let $f_{N}(\mathcal{C})$ denote the size of an optimal coin purse for a coin set $\mathcal{C}$ used in making change for transactions in $\{1,2, \ldots, N \}$, and
\begin{equation*}
f_{N}^{\ast} := \min_{\mathcal{C}} f_{N}(\mathcal{C})
\end{equation*}
the size of a corresponding minimal coin purse.

\begin{proposition}
\label{prop_coin_purse}
$f_{N}^{\ast}= \lceil \log_{2}(N+1) \rceil$.
\end{proposition}

\begin{proof}
Let $\mathcal{C}$ be a coin set.
For a list $\ell = (\ell_{k})_{1 \le k \le n}$ with $\ell_{k} \in \mathcal{C}$, let
\begin{equation*}
\mathcal{S}(\ell) = \left \{ \sum_{k=1}^{n} \beta_{k} \ell_{k} : \beta_{k} \in \{0,1 \}, \, \exists \beta_{k} \ne 0\right \}.
\end{equation*}
Then $\lvert \mathcal{S}(\ell) \rvert \le 2^{n} -1$.
Moreover, $f_{N}^{\ast}$ is the length of a minimal list $\ell$ such that $\mathcal{S}(\ell) \supset \{1, \ldots, N \}$.
Thus $f_{N}^{\ast} \ge \lceil \log_{2} (N+1) \rceil$ and the binary coin set $\mathcal{C}_{0} = \{1,2, \ldots, 2^{ \lceil \log_{2} (N+1) \rceil -1} \}$ achieves this bound.
\end{proof}

For $N=99$ this is $f_{99}^{\ast} = 7$ achieved by $\{1,2,4,8,16,32,64\}$ using $3.19$ coins on average.
The optimal, unrestricted, coin set of cardinality $7$ uses $2.68$ coins on average \cite{Shallit2003}.

\section{Complete greediness}

Let $\mathcal{C}$ be a coin set.
In addition to greediness we shall consider uniqueness.
$\mathcal{C}$ is said to be \emph{unique} if every positive integer has only one minimal representation, e.g., $\{1,2,3\}$ is not unique $4 = 2 +2 = 3 +1$, for $\lvert \mathcal{C} \rvert \le 3$ see \cite{Bryant2012}.
Total and complete uniqueness of coin sets are defined in the same manor as for greediness.
$\text{opt}(\cdot \, ; \, \mathcal{C} )$ and $\text{grd}(\cdot \, ; \, \mathcal{C})$ denote the \emph{optimal and greedy cost functions}, i.e. number of elements in an optimal and the greedy representations under $\mathcal{C}$.

A few remarks about greediness of coin sets.
Greediness does not imply total greediness \cite{Cowen2008}, $\{1,2,4,5,8\}$ is greedy but $\{1,2,4,5\}$ is not, does not imply complete greediness, $\{1,3,5,7\}$ is totally greedy but $\{1,5,7\}$ is not greedy.
All cardinality $2$ coin sets are greedy and if $\mathcal{C} = \{c_{1},\ldots, c_{n}\}$ is not greedy the least violator is in $(c_{3}+1,c_{n-1}+c_{n})$ \cite{Kozen1994}.

\begin{theorem}{\rm \cite{Magazine1975,Cowen2008}}
\label{thm_opt}
Let $\mathcal{C} = \{c_{1},\ldots,c_{n}\}$ be a greedy coin set, $c_{n+1} > c_{n}$ and $m = \lceil \frac{c_{n+1}}{c_{n}} \rceil$.
Then $\mathcal{C}' = \mathcal{C} \cup \{c_{n+1}\}$ is greedy if and only if $\mathrm{opt}(mc_{n} \, ; \, \mathcal{C}') = \mathrm{grd}(mc_{n} \, ; \, \mathcal{C}')$ if and only if $\mathrm{grd}(mc_{n} \, ; \, \mathcal{C}' ) \le m$.
\end{theorem}

\begin{corollary}{\rm \cite{Kozen1994,Cowen2008}}
\label{cor_3pt}
A coin set $\mathcal{C} = \{1,a,b\}$ is greedy if and only if $b \ge \lceil \frac{b}{a} \rceil (a-1) + 1$.
\end{corollary}

\begin{proposition}
\label{prop_3pt_gc}
Let $\mathcal{C} = \{c_{1}, \ldots, c_{n} \}$ be a coin set.
The following are equivalent.
\begin{enumerate}[(i)]
\item
$\mathcal{C}$ is completely greedy;
\item
All sub-coin sets of cardinality $3$ are greedy;
\item
For all $k \ge 3$, $\{c_{1},c_{k-1},c_{k} \}$ is greedy.
\end{enumerate}
\end{proposition}

\begin{proof}
If $\lvert \mathcal{C} \rvert = 1$ or $2$, then $\mathcal{C}$ is completely greedy and $(ii)$, $(iii)$ are vacuously true.
Suppose $\lvert \mathcal{C} \rvert \ge 3$.
We begin with a technical result.
For $a,b \ge 0$, $\lceil ab \rceil \le \lceil a \rceil \lceil b \rceil$.
$(iii) \Longrightarrow (ii)$
Let $\{ 1,c_{i},c_{j} \}$ be a sub-coin set.
Then by assumption, for all $i < k \le j$, $\{1,c_{k-1},c_{k}\}$ is greedy and by Corollary \ref{cor_3pt} $c_{k} \ge \lceil \frac{c_{k}}{c_{k-1}} \rceil(c_{k-1}-1) +1$.
Thus
\begin{equation*}
c_{j} 
\ge \left \lceil \frac{c_{j}}{c_{j-1}} \right \rceil ( c_{j-1} -1) +1 
\ge \left( \prod_{k=i+1}^{j} \left \lceil \frac{c_{k}}{c_{k-1}} \right \rceil \right)( c_{i} -1) + 1
\ge \left \lceil \frac{c_{j}}{c_{i}} \right \rceil ( c_{i} -1) + 1,
\end{equation*}
where the last inequality is the beginning result.
Hence $\{1,c_{i},c_{j}\}$ is greedy by Corollary \ref{cor_3pt}.
$(ii) \Longrightarrow (i)$
Let $\mathcal{C}_{k} = \{c_{1}, \ldots, c_{k} \}$.
$\mathcal{C}_{1}$ is completely greedy.
Suppose $\mathcal{C}_{k-1}$ is.
Let $\mathcal{S}$ be a sub-coin set of $\mathcal{C}_{k}$ containing $c_{k}$, otherwise $\mathcal{S}$ is greedy by the inductive statement.
Thus $\mathcal{S} \backslash \{c_{k}\}$ is greedy.
Let $c^{\ast} = \max \{ c : c \in \mathcal{S} \backslash \{c_{k} \} \}$ and $m = \lceil \frac{ c_{k} }{ c^{\ast} } \rceil$.
Then $\text{grd}( m c^{\ast} \, ; \, \mathcal{S} ) \le \text{grd}(mc^{\ast} \, ; \{1,c_{k} \}) = 1 + mc^{\ast} - c_{k} \le m$ by greediness of $\{1,c^{\ast},c_{k}\}$ and Corollary \ref{cor_3pt}.
Hence $\mathcal{S}$ is greedy by Theorem \ref{thm_opt}.
\end{proof}

\begin{lemma}
\label{lemma_bin_lb}
Let $\mathcal{C} = \{c_{1},\ldots, c_{n} \}$ be a coin set.
If $\mathcal{C}$ is completely greedy and totally unique, then, for all $1\le k \le n$ and $0 \le i \le k - 1$, $c_{k} \ge 2^{i}c_{k-i}$.
\end{lemma}

\begin{proof}
Obvious for $c_{1} =1$ and $c_{2} \ge 2$.
For all $k \ge 3$, by complete greediness, $\{1,c_{k-1},c_{k}\}$ is greedy.
Thus $c_{k} \ge 2c_{k-1} - 1$ by Corollary \ref{cor_3pt}, but $\{1, \ldots, c_{k-1}, 2 c_{k-1} -1\}$ is not unique at $2 c_{k-1}$.
Hence $c_{k} \ge 2 c_{k-1}$.
The rest follows by induction.
\end{proof}

For a coin set $\mathcal{C} = \{c_{1}, \ldots, c_{m} \}$, let
\begin{equation*}
n(N) := \max \{ n : c_{n} \le N \}.
\end{equation*}
If $\mathcal{C}$ is completely greedy and totally unique then $c_{n} \ge 2^{n-1}$ and $n(N) \le \lceil \log_{2}(N+1) \rceil$.
If the uniqueness is dropped $c_{n} \ge 2^{n-2} + 1$, for $n \ge 2$, and $n(N) \le \lfloor \log_{2}(N-1) \rfloor + 2$, for $N \ge 2$. 

\begin{proposition}
\label{prop_div_unique}
Let $\mathcal{C} = \{c_{1},\ldots, c_{n}\}$ be a coin set.
If $c_{k} \vert c_{k+1}$, for all $1 \le k \le n-1$, then $\mathcal{C}$ is completely greedy and completely unique.
\end{proposition}

\begin{proof}
By Proposition \ref{prop_3pt_gc} and Corollary \ref{cor_3pt}, $\mathcal{C}$ is completely greedy.
It suffices to show that $\mathcal{C}$ is unique.
Suppose not.
Let $N$ be the least violator and $c_{a}$ its greatest lower bounding coin.
Then $N= \sum_{k=1}^{b} \alpha_{k} c_{k}$ for some $b< a$, otherwise $N - c_{a}$ would be a lesser violator.
As the representation is minimal and $c_{k} \lvert c_{k+1}$, for all $1 \le k \le b$, $\alpha_{k} c_{k} < c_{k+1}$ implies $(\alpha_{k}+1) c_{k} \le c_{k+1}$ and $\sum_{i=1}^{k} \alpha_{i} c_{i} < c_{k+1}$.
Hence $\sum_{k=1}^{b} \alpha_{k} c_{k} < c_{b+1} \le c_{a} \le N$.
\end{proof}

The binary coin set is the only completely greedy totally unique coin set using one coin of each denomination to make change.
Let $\mathcal{C} = \{c_{1}, \ldots, c_{n} \}$ be a coin set.
Then, for $N \ge c_{n}$,
\begin{equation}
\label{eq_greedy_coin_purse}
f_{N}(\mathcal{C}) \ge \left \lfloor \frac{N}{c_{n}} \right \rfloor + \sum_{k=1}^{n-1} \left \lfloor \frac{ c_{k+1} - 1 }{ c_{k} } \right \rfloor.
\end{equation}
Make change at $\lfloor \frac{N}{c_{n}} \rfloor c_{n}$ or $\lfloor \frac{ c_{k+1} - 1 }{ c_{k} } \rfloor c_{k} $.
If $\mathcal{C}$ is greedy, then this is tight.
Thus for a greedy coin set $\lfloor \frac{ N}{ c_{n}} \rfloor$ copies of $c_{n}$ are needed and $\lfloor \frac{ c_{k+1} -1}{c_{k} } \rfloor$ of $c_{k}$.
If $\mathcal{C}$ is completely greedy and totally unique $c_{k+1} \ge 2 c_{k}$ by Lemma \ref{lemma_bin_lb} and for any $k$ strictly exceeding this bound $c_{k}$ must be used more than once.

\section{Asymptotics}

Let
\begin{equation*}
\text{cost}(N \, ; \, \mathcal{C} )
:= \sum_{k=1}^{N} \text{opt}(k \, ; \, \mathcal{C})
= \sum_{k=1}^{N} \text{grd}(k \, ; \, \mathcal{C}),
\end{equation*}
where the last equality holds if $\mathcal{C}$ is greedy, and $\text{cost}(0 \, ; \, \mathcal{C}) := 0$.
A coin set is said to be infinite if $n(N) \rightarrow \infty$.
An infinite coin set is said to be (completely) greedy/unique if the first $n$ terms are (completely) greedy/unique for any finite $n$.
Let $\mathcal{B} := \{1,2, \ldots, 2^{n}, \ldots \}$ denote the infinite binary coin set.
Then \cite{Delange1975}
\begin{equation*}
\text{cost}( N \, ; \, \mathcal{B})
= \frac{1}{2} N \log_{2} N + N F\left( \log_{2} N \right),
\end{equation*}
where $F$ is a continuous nonpositive nowhere differentiable function of period $1$.
Thus
\begin{equation*}
\lim_{N \rightarrow \infty} \frac{ \text{cost}( N \, ; \, \mathcal{B}) }{ n(N) N } 
= \frac{1}{2},
\end{equation*}
where $n(N) = \lceil \log_{2}(N+1) \rceil$.

\begin{theorem}
\label{theorem_icg_lb}
Let $\mathcal{C}$ be an infinite completely greedy coin set.
Then
\begin{equation*}
\limsup_{N \rightarrow \infty} \frac{ \mathrm{cost}(N \, ; \, \mathcal{C}) }{ \left( \lfloor \log_{2}(N-1) \rfloor + 2 \right) N } \ge \frac{1}{2}.
\end{equation*}
\end{theorem}

\begin{proof}
Dropping to a subsequence
\begin{equation*}
\limsup_{N \rightarrow \infty} \frac{ \mathrm{cost}(N \, ; \, \mathcal{C}) }{ \left( \lfloor \log_{2}(N-1) \rfloor +2 \right) N }
\ge \limsup_{n \rightarrow \infty} \frac{ \mathrm{cost}( c_{n+1} -1 \, ; \, \mathcal{C}) }{ \left( \lfloor \log_{2}\left( c_{n+1}-2\right) \rfloor + 2 \right) (c_{n+1}-1) }.
\end{equation*}
By Lemma \ref{lemma_sum_lb}, $\mathrm{cost}(c_{n+1} -1 \, ; \, \mathcal{C} ) \ge \frac{1}{2} c_{n+1} \log c_{n+1} + o(n c_{n+1})$ and, for a completely greedy coin set, $n = n(c_{n+1} -1) \le \lfloor \log_{2}(c_{n+1}-2) \rfloor +2$.
\end{proof}

Let $\mathcal{C} = \{c_{1},\ldots, c_{n}\}$ be a greedy coin set, then, for $q \in \mathbb{N}$ and $0 \le r < c_{n}$,
\begin{align*}
\text{cost}(qc_{n} + r \, ; \, \mathcal{C})
&= \frac{q(q-1)}{2}c_{n} + q \cdot \text{cost}(c_{n}-1 \, ; \, \mathcal{C} \backslash \{ c_{n} \}) + q(r+1) \\
&\quad + \text{cost}( r \, ; \, \mathcal{C} \backslash \{c_{n} \} ).
\end{align*}
In particular,
\begin{equation*}
\text{cost}(qc_{n} -1 \, ; \, \mathcal{C})
= \frac{q(q-1)}{2} c_{n} + q \cdot \text{cost}( c_{n} -1 \, ; \, \mathcal{C} \backslash \{ c_{n} \} ),
\end{equation*}
where by abusive of notation $\text{cost}( -1 \, ; \, \mathcal{C}) : = 0$, and thusly
\begin{align}
\label{eq_greedy_quadratic}
\text{cost}(qc_{n} + r -1  \, ; \, \mathcal{C})
&= \frac{q(q-1)}{2}c_{n} + q \cdot \text{cost}(c_{n}-1 \, ; \, \mathcal{C} \backslash \{ c_{n} \}) + qr \\ \nonumber
&\quad + \text{cost}( r -1 \, ; \, \mathcal{C} \backslash \{c_{n} \} ),
\end{align}
for all $0 \le r < c_{n}$.

\begin{lemma}
\label{lemma_greedy_oN}
Let $\mathcal{C} = \{c_{1}, \ldots, c_{n}, \ldots \}$ be an infinite completely greedy coin set.
Then
\begin{equation*}
\mathrm{grd}(N \, ; \, \mathcal{C}) = o(N).
\end{equation*}
\end{lemma}

\begin{proof}
For all $k$, $c_{k+1} \ge 2c_{k} -1$ by Corollary \ref{cor_3pt}.
Inductively this implies $c_{k} \ge 2^{k-2}$ and $c_{k} \le 2^{-n+k} N + 1$, for all $1 \le k \le n$, where $n=n(N)$.
By (\ref{eq_greedy_coin_purse}), as equality holds for a greedy coin set,
\begin{align*}
\mathrm{grd}(N \, ; \, \mathcal{C})
&\le f_{N}(\mathcal{C}) \\
&= \left \lfloor \frac{N}{c_{n}} \right \rfloor + \sum_{k=1}^{n-1} \left \lfloor \frac{ c_{k+1} - 1 }{ c_{k} } \right \rfloor \\
&\le \left \lfloor \frac{N}{ 2^{n-2}} \right \rfloor + \sum_{k=1}^{n-1} \left \lfloor \frac{ 2^{-n+k+1} N }{ 2^{k-2} } \right \rfloor \\
&\le n2^{-n+3}N.
\end{align*}
\end{proof}

\begin{lemma}
\label{lemma_sum_lb}
Let $\mathcal{C} = \{c_{1}, \ldots, c_{n}, \ldots \}$ be an infinite completely greedy coin set, $c_{n+1} = q_{n} c_{n} + r_{n}$, where $q_{n}$ and $r_{n}$ are chosen by the division algorithm, $a_{n} = \lceil \frac{c_{n+1}}{c_{n}} \rceil$ and $s_{n} = \sum_{k=1}^{n} (b_{k} - 1)$, where $b_{1} = a_{1} = q_{1}$ and, for $n \ge 2$,
\begin{equation*}
b_{n} = 
\begin{cases}
a_{n} & q_{n} \le s_{n-1} \\
q_{n} & \text{else}
\end{cases}.
\end{equation*}
Then
\begin{equation*}
\mathrm{cost}(c_{n+1} -1 \, ; \, \mathcal{C} ) \ge \frac{s_{n}}{2} c_{n+1} + o(n c_{n+1}).
\end{equation*}
Moreover,
\begin{equation*}
s_{n} \ge \max \{ n, \log_{2} c_{n+1} -1 \}.
\end{equation*}
\end{lemma}

\begin{proof}
It suffices to show that
\begin{equation*}
\mathrm{cost}(c_{n+1} -1 \, ; \, \mathcal{C} ) \ge \frac{s_{n}}{2} c_{n+1} - t_{n}o(1),
\end{equation*}
where $t_{n}$ is at most $O(nc_{n+1})$.
Let $t_{1} = 0$ and, for $n \ge 2$,
\begin{equation*}
t_{n} = 
\begin{cases}
(q_{n} + 1) t_{n-1} + q_{n} c_{n} & r_{n} > 0 \text{ and } q_{n} \le s_{n-1} \\
q_{n} t_{n-1} & \text{else}
\end{cases}.
\end{equation*}
For the order condition eventually $s_{n-1} \le \frac{c_{n}}{n}$, for the details see Lemma \ref{lemma_prod_sum_lb}.
Let $n_{0}$ be such an $n$, then $t_{n_{0}} \le M n_{0} c_{n_{0} + 1}$ for some $M \ge 1$.
Let $n > n_{0}$ and suppose $t_{n-1} \le M(n-1)c_{n}$.
If $r_{n} > 0$, then by Corollary \ref{cor_3pt}, as $\{1,c_{n},c_{n+1}\}$ is greedy,
\begin{equation*}
q_{n} c_{n} + r_{n} = c_{n+1} \ge (q_{n} + 1) (c_{n} -1) + 1
\qquad \Longrightarrow \qquad
r_{n} \ge c_{n} - q_{n}.
\end{equation*}
Thus
\begin{equation*}
q_{n} \le s_{n-1} \le \frac{c_{n}}{n}
\qquad \Longrightarrow \qquad
(n-1)c_{n} 
\le n(c_{n} - q_{n})
\le n r_{n}
\end{equation*}
and
\begin{align*}
t_{n} 
&= (q_{n}+1) t_{n-1} + q_{n} c_{n} \\
&\le (q_{n}+1) M (n-1) c_{n} + q_{n} c_{n} \\
&\le M n (q_{n} c_{n} + r_{n}) \\
&= M n c_{n+1}.
\end{align*}
For $c_{2}$, $r_{1}= 0$ thus
\begin{equation*}
\text{cost}(c_{2}-1 \, ; \, \mathcal{C}) 
= \frac{q_{1}(q_{1}-1)}{2} 
= \frac{ s_{1} }{2} c_{2}.
\end{equation*}
Suppose for $n$.
Let $\mathcal{C}_{n} = \{c_{1}, \ldots, c_{n} \}$.
Then by (\ref{eq_greedy_quadratic})
\begin{align}
\label{eq_cnp1_lb}
\nonumber
&\text{cost}( c_{n+1} -1  \, ; \, \mathcal{C}) \\ \nonumber
&\quad = \text{cost}(q_{n} c_{n} + r_{n} -1  \, ; \, \mathcal{C}_{n}) \\ \nonumber
&\quad = \frac{q_{n}(q_{n}-1)}{2}c_{n} + q_{n} \cdot \text{cost}(c_{n}-1 \, ; \, \mathcal{C}_{n-1}) + q_{n}r_{n} + \text{cost}( r_{n} -1 \, ; \, \mathcal{C}_{n-1} ) \\ \nonumber
&\quad \ge \frac{q_{n}(q_{n}-1)}{2}c_{n} + q_{n} \frac{s_{n-1}}{2} c_{n} - q_{n} t_{n-1} o(1) + q_{n}r_{n} + \text{cost}( r_{n} -1 \, ; \, \mathcal{C}_{n-1} ) \\ \nonumber
&\quad = \frac{ q_{n} -1 }{2} c_{n+1} - \frac{ q_{n} -1 }{2} r_{n} + \frac{s_{n-1}}{2} c_{n+1} - \frac{s_{n-1}}{2}r_{n} - q_{n} t_{n-1} o(1) + q_{n} r_{n} \\ \nonumber
&\qquad + \text{cost}(r_{n}-1 \, ; \, \mathcal{C}_{n-1} )  \\
&\quad = \frac{ s_{n} }{2} c_{n+1} - \frac{b_{n} - q_{n} }{2} c_{n+1} + \frac{ q_{n} + 1 - s_{n-1} }{2} r_{n} + \text{cost}(r_{n} - 1 \, ; \, \mathcal{C}_{n-1} ) \\ \nonumber
&\qquad - q_{n} t_{n-1}o(1).
\end{align}
The desired bound holds for $r_{n} = 0$.
Suppose $r_{n} > 0$.
If $q_{n} \le s_{n-1}$, then $b_{n} - q_{n}= 1$ and
\begin{align*} 
&\text{cost}( c_{n+1} -1  \, ; \, \mathcal{C}) \\
&\quad \ge \frac{ s_{n} }{2} c_{n+1} - \frac{1}{2}(q_{n} c_{n} + r_{n} ) + \frac{ q_{n} + 1 - s_{n-1} }{2} r_{n} + \text{cost}(r_{n} \! - \! 1 \, ; \, \mathcal{C}_{n-1} ) \\
&\qquad - q_{n}t_{n-1} o(1) \\
&\quad = \frac{ s_{n} }{2} c_{n+1} - \frac{1}{2}q_{n} c_{n} + \frac{ q_{n} - s_{n-1} }{2} r_{n} + \text{cost}(r_{n} \! - \! 1 \, ; \, \mathcal{C}_{n-1} ) - q_{n}t_{n-1}o(1).
\end{align*}
Moreover, by Corollary \ref{cor_3pt} $r_{n} \ge c_{n} - \min \{ q_{n}, c_{n} -1 \}$.
Thus by Lemma \ref{lemma_greedy_oN}
\begin{align*}
\text{cost}( r_{n} - 1 \, ; \, \mathcal{C}_{n-1})
&\ge \text{cost}( c_{n} - \min \{ q_{n}, c_{n} -1\} -1 \, ; \, \mathcal{C}_{n-1}) \\
&= \text{cost}( c_{n} -1 \, ; \, \mathcal{C}_{n-1}) - \sum_{i=0}^{\min \{ q_{n}, c_{n} -1 \} -1} \text{grd}(c_{n} - i - 1 \, ; \mathcal{C}_{n-1} )\\
&= \text{cost}( c_{n} -1 \, ; \, \mathcal{C}_{n-1}) - \sum_{i=0}^{\min \{ q_{n}, c_{n} -1 \} -1} o(c_{n} - i - 1) \\
&\ge \frac{s_{n-1}}{2}c_{n} - t_{n-1} o(1) - q_{n} o(c_{n})\\
&=\frac{s_{n-1}}{2}c_{n} - t_{n-1}o(1) - q_{n} c_{n} o(1) \\
&= \frac{1}{2}q_{n}c_{n} +\frac{s_{n-1} - q_{n} }{2} c_{n} - (t_{n-1} + q_{n}c_{n})o(1) \\
&\ge \frac{1}{2}q_{n}c_{n} +\frac{s_{n-1} - q_{n}}{2} r_{n} - (t_{n-1} + q_{n} c_{n})o(1).
\end{align*}
Combining gives the desired bound, where $q_{n} t_{n-1} + t_{n-1} + q_{n} c_{n} = t_{n}$.
If $q_{n} > s_{n-1}$, then $b_{n} = q_{n}$ and all terms in (\ref{eq_cnp1_lb}) are nonnegative.
For the lower bounds on $s_{n}$.
If $q_{n} =1$, then $q_{n} \le c_{2} -1 = q_{1} - 1 = s_{1} \le s_{n-1}$.
Thus $b_{n} = a_{n} \ge 2$ and $s_{n} \ge n$.
Similarly, $s_{1} = q_{1} - 1 \ge \log_{2} q_{1} = \log_{2} c_{2}$, where $\log_{2} x \le x -1 $ for $x \ge 2$.
If $r_{n} = 0$
\begin{equation*}
\log_{2} c_{n+1}
=\log_{2} q_{n} + \log_{2} c_{n}
\le q_{n} -1 + \log_{2} c_{n}.
\end{equation*} 
If $q_{n} \ge 3$
\begin{equation*}
\log_{2} c_{n+1}
\le \log_{2}(q_{n} +1) + \log_{2} c_{n}
\le q_{n} -1 + \log_{2} c_{n},
\end{equation*}
where $\log_{2}x \le x -2$ for $x \ge 4$.
Suppose $q_{n} \le 2$ and $r_{n} > 0$.
By Corollary \ref{cor_3pt} $c_{n+1} = 2c_{n}-1,3c_{n}-1,3c_{n}-2$.
For all $n$, 
\begin{equation*}
\log_{2} c_{n+1}
\le \log_{2} a_{n} + \log_{2} c_{n}
\le a_{n} -1 + \log_{2} c_{n}.
\end{equation*}
As $s_{n} \ge n$, the only case for which $b_{n} \ne a_{n}$ is $c_{2} -1 = s_{1} < q_{2} = 2$ or $c_{2} = 2$ and $c_{3} = 5$.
Hence $s_{n} < \log c_{n+1}$ at most once and by at most $1$.
\end{proof}

\begin{lemma}
\label{lemma_prod_sum_lb}
Let $\mathcal{C} = \{c_{1}, \ldots, c_{n}, \ldots \}$ be an infinite completely greedy coin set, $a_{n} = \lceil \frac{c_{n+1}}{c_{n}} \rceil$ and $u_{n} = \sum_{k=1}^{n} (a_{k} - 1)$.
There exists $n_{0}$ such that for all $n \ge n_{0}$
\begin{equation*}
c_{n} \ge n u_{n-1}.
\end{equation*}
\end{lemma}

\begin{proof}
Let $n \ge 3$ and
\begin{equation*}
g_{n}( \underline{a})
=\left( \prod_{k=2}^{n-1} a_{k} \right) (a_{1} -1) + 1 - n\sum_{k=1}^{n-1} (a_{k} - 1).
\end{equation*}
By Corollary \ref{cor_3pt} $c_{n} - n u_{n-1} \ge g_{n}( \underline{a})$, where $a_{1} = c_{2}$.
Then
\begin{equation*}
\frac{ \partial }{ \partial a_{1}}g_{n}(\underline{a})
= \left( \prod_{k=2}^{n-1} a_{k} \right) - n,
\end{equation*}
and, for $i \ne 1$,
\begin{equation*}
\frac{ \partial }{ \partial a_{i}}g_{n}(\underline{a})
= \left( \prod_{k=2,k\ne i}^{n-1} a_{k} \right)(a_{1}-1) - n.
\end{equation*}
Thus, for all $i$,
\begin{equation*}
\frac{ \partial }{ \partial a_{i}} g_{n}(\underline{a})
\ge 2^{n-3} - n.
\end{equation*}
Let $n_{1}$ be such that $2^{n-3} \ge n$ for all $n \ge n_{1}$.
Then, for all $n \ge n_{1}$, $g_{n}$ is nondecreasing in each coordinate and it suffices to show that the bound holds for $a_{i} = 2$.
Choose $n_{2}$ such that $2^{n-2} + 1 \ge n(n-1)$ for all $n \ge n_{2}$.
Let $n_{0} = \max\{3,n_{1},n_{2} \}$.
For all $n \ge n_{0}$, $c_{n} \ge n u_{n-1}$.
\end{proof}

\end{document}